 \documentclass[]{article}

\usepackage{amsmath,amsthm, amssymb}
\usepackage{graphicx}
\usepackage{color}

\newtheorem{thm}{Theorem}
\newtheorem{lem}{Lemma}[section]
\newtheorem{cor}[lem]{Corollary}

\theoremstyle{definition}
\newtheorem{defn}[lem]{Definition}
\newtheorem{exe}[lem]{Example}
\newtheorem{algo}{Algorithm}

\theoremstyle{remark}
\newtheorem{rem}[lem]{Remark}
\newtheorem*{rem*}{Remark}

\newcommand{\ttlt}{\textrm{lt}}

\newcommand{\ttdeg}{\textrm{deg}}

\newcommand{\tndeg}{\textnormal{deg}}

\newcommand{\LX}{\mathrm{L}_\prec^{\scriptscriptstyle\!\![X]}}

\newcommand{\I}{\mathfrak{I}}
\newcommand{\F}{\Omega_{u}}

\newcommand{\redGB}{\underline{GB}}

\newcommand{\N}{\mathbb{N}}
\newcommand{\R}{\mathbb{R}}

\newcommand{\K}{\mathbb{K}}

\newcommand{\zero}{\mathbf{0}}
\newcommand{\wrt}{w.r.t.}

\begin{document}


\title{Ideal-specific elimination orders form a star-shaped region}

\author{Hartwig Bosse\footnote{J.W. Goethe Universit\"at, Frankfurt, bosse@math.uni-frankfurt.de }\ ,%
  \setcounter{footnote}{6}%
\ Christine G\"artner\footnote{Freie Universit\"at Berlin, christine.gaertner@math.fu-berlin.de},\hspace*{8pt} and%
\ Oleg Golubitsky\footnote{Google, Inc., Waterloo, oleg.golubitsky@gmail.com}
}

\maketitle

\begin{abstract}
This paper shows that for any given polynomial ideal $\I\subset\K[x_1,\ldots,x_n]$ the collection of Gr\"obner cones corresponding 
to $\I$-specific elimination orders form a star-shaped region which contrary to first intuition in general is not convex.
\smallskip

Moreover we show that the corresponding region may contain Gr\"obner cones intersecting in the boundary of the Gr\"obner fan in the origin only. This implies that Gr\"obner walks aiming for the elimination of variables from a polynomial ideal can be terminated earlier than previously known. We provide a slightly improved stopping criterion for a known Gr\"obner walk algorithm for the elemination of variables. 
\end{abstract}



\section{Introduction}

Elimination in systems of polynomial equations is a classical topic important in optimization and modeling. 
Given an ideal $\I$ of polynomials in $\K[X][U]:=\K[x_1,\ldots,x_n,u_1,\ldots,u_m]$ over some field $\K$, the task of eliminating the variables $u_i$ can be solved by finding an ideal basis for the the so called \emph{elimination ideal} $\I\cap \K[X]$, where  $\K[X]=\K[x_1,\ldots,x_n]$.  
This can be achieved using resultants 
(see \cite{book::syl}, \cite{book::sal}, or \cite{article::sed}) 
or by calculating a  Gr\"obner basis (GB) for $\I$ with respect to some special monomial order (see \cite{proceedings::Buchberger}, \cite{proceedings::Kalkbrenner}), as for example, the pure lexicographic or block term orders.
Concerning these approaches, the
method using Gr\"obner bases has some important advantages, namely, the method is reliable and can
algorithmically solve the problem in full generality. 
\smallskip

In the Gr\"obner basis approach one calculates a  Gr\"obner basis $G_{\prec_\mathrm{elim}}$ with respect to a suitable monomial order $ \prec_\mathrm{elim}$, 
such that those polynomials in $G_{\prec_\mathrm{elim}}\cap\K[X] $ form a Gr\"obner basis for  $\I\cap \K[X]$. 
Calculating these very specific Gr\"obner bases directly can in practice be rather difficult.
One way to overcome this, is to calculate such a special GB 
by performing a Gr\"obner walk, a method introduced by Collart, Kalkbrener, and Mall in \cite{article::Collart97convertingbases}.

The actual walk consists of a series of elementary GB-conversions which are easy to compute.
Starting with some easily computable GB of $\mathcal I$ with respect to some order $\prec_\mathrm{start}$, 
step-by-step, intermediate GBs for orders in between $\prec_\mathrm{start}$ and $\prec_\mathrm{elim}$ are calculated. Each basis-conversion from one intermediate GB into the next is (in general) relatively cheap computationwise, 
keeping the overall amount of necessary calculations relatively low (see \cite{article::AGK3}).

\medskip

To handle the intermediate orders in any Gr\"obner walk algebraically, 
one represents them by weight vectors and introduces the concept of a Gr\"obner fan:

For a fixed ideal $\I\subset \mathbb{K}[X][U]$, any proper monomial order for monomials in $\mathbb{K}[X][U]$ can be represented by some weight vector in $\omega\in\R_{\ge 0}^{n+m}$.  The Gr\"obner fan, introduced by Mora and Robbiano in \cite{article::mora_robbiano}, is a polyhedral complex, which subdivides the weight vectors in $\R_{\ge 0}^{n+m}$. Each cell of the Gr\"obner fan is an equivalence class of such weight vectors:

Two weight vectors are equivalent, if the monomial order they represent yields the same Gr\"obner basis for $\I$.
The closure of such an equivalence class is a \emph{Gr\"obner cone} and the collection of these cones forms the Gr\"obner fan. Note that Gr\"obner cones are convex polyhedral cones (see  \cite{article::mora_robbiano}).
\medskip

Concerning Gr\"obner walks used in elmination of variables,
Tran proposes in  \cite{article::Tran1} to have the target monomial order $\prec_{\mathrm{elim}}$ dependent on $\I$, combining the Gr\"obner walk technique with a sudden-death-algorithm.

So instead of using \emph{the same} elimination
term order for  \emph{all} ideals, Tran proposes to use \emph{an ideal-specific} monomial order suitable (only) for elimination in the specifically given ideal. He characterizes these special ideal-specific orders via the corresponding reduced Gr\"obner basis.

In addition to being faster on some examined test bed cases, his approach gets rid of several algebraic technicalities usually involved in Gr\"obner walks, e.g. his approach simplifies the necessary perturbation of the weight vector representing the elimination order:
\smallskip

Gr\"obner walk algorithms are particularly fast, if the given path of the walk is \emph{generic}.
To achieve this, one has to perturb the target weight vector of the walk in a suitable manner (see e.g. \cite{article::Fukuda_Jensen_Thomas}).
In \cite{article::Tran1}, Tran observed that using ideal-specific elimination orders, it suffices to end a Gr\"obner walk in a Gr\"obner cone adjacent to some \emph{elimination vector} (see below) which eases the requirements on the necessary perturbations.
\smallskip

We refine Tran's findings by giving a  more precise classification of those Gr\"obner cones, 
which correspond to ideal-specific elimination orders.

\medskip

\subsection{Main result}

The main results of this paper are the following:

For a given ideal $\I\subseteq \K[x_1,\ldots,x_n,u_1,\ldots, u_m]$, 
the union of all Gr\"obner cones belonging to $\I$-specific orders for the elimination of $u_1,\ldots,u_m$ from $\I$
form a star-shaped region with center $\F:=\{\omega\in\R^{n+m}_{\ge0}\;:\;\omega_1=0,\ldots,\omega_n=0\}$.
This means that if one wishes to eliminate the variables $u_i$ from $\I$, i.e., 
one wants to calculate some Gr\"obner basis for $\I\bigcap \K[X]$, the orders $\prec$ that \emph{do yield} such a Gr\"obner basis have Gr\"obner cones, whose union is a star-shaped region with center $\F$.

Moreover we show that (for some ideals $\I$) some of the Gr\"obner cones which belong to $\I$-specific elimination orders intersect the boundary of the Gr\"obner fan in the point zero \emph{only}, 
meaning that for such cones all points but the vertex lie in the relative interior of the Gr\"obner fan.
\smallskip

Both results are very useful when trying to eliminate variables using the Gr\"obner walk-approach:
First of all, we can improve the stopping criterion for such a Gr\"obner walk relative to the known result of Tran \cite{article::Tran}.
Moreover, knowing the geometric shape of the target-region can help improve the step-decision process in a Gr\"obner walk towards an elimination-basis.

Finally, in the general case, just as shown by Tran, using our algorithm, one can get rid of technicalities involved in the implementation of the Gr\"obner walk such as the perturbation of the target vector (see \cite{article::Tran}).

\section{Notation}
In the following we introduce some general notation for polynomials and monomial orders.
To avoid clashes with our distinct variables $x_i$ and $u_j$, here we name all variables 
$y_i$, assuming $(y_1,\ldots,y_{n+m})=(x_1,\ldots,x_n,u_1,\ldots, u_m)$.
So in the following we consider polynomials $f=\sum_{\alpha}f_{\alpha}y^{\alpha}$ 
where $y^\alpha:=\prod_{i=1}^{n+m}y_i^{\alpha_i}$ is a monomial with exponent $\alpha\in\N^{n+m}$ and 
the coefficients $f_{\alpha}$ are from some field $\K$.

\subsection{Monomial orders }
In the following let $\prec$ be some monomial order and $f,g\in \K[Y]$.
We denote the \emph{leading term} of $f$ \wrt\ $\prec$ by  $\ttlt_{\prec}(f)$.
Let $\I\subset\K[Y]$ be some polynomial ideal, then the \textit{initial ideal} of $\I$ \wrt\ $\prec$ is the ideal
$\left\langle \ttlt_{\prec}(\I)\right\rangle$ 
which is generated by the set of \emph{leading terms} of $\I$, i.e., $\ttlt_{\prec}(\I):=\left\{ \ttlt_{\prec}(f)\;:\; f\in \I\right\}$. 
\smallskip

\subsection{Reduced Gr\"obner bases}

In this work we consider reduced Gr\"obner bases:
Let $\I\subseteq \K[Y]$ be some monomial ideal and let $\prec$ be some monomial order.
A Gr\"obner basis $G$ for $\I$ \wrt\ $\prec$ is called \emph{reduced} if for every pair $g,h\in G$, $g\ne h$  one has that $\ttlt_{\prec}(g)$ does not divide any monomial of $h$ (so $h$ can not be reduced by $g$ any further).
Moreover $G$ is called \emph{normed} if for all $g\in G$ the leading coefficient is $1$. 

Every ideal $\I\subseteq \K[Y]$ has a unique finite normed \underline{reduced} Gr\"obner basis with respect to $\prec$ (see \cite{book::cox_little_oshea}, \cite{phd::buchberger}), which we denote by $\redGB(\I,\prec)$. 
\smallskip


\subsection{Weight vectors}

To algebraically work with monomial orders, it is helpful to represent them by weight vectors:
The the set of all \textit{weight vectors} $\Omega:=\R^{n+m}_{\geq 0}$ is the non-negative orthant. 
Let $f\in \K[Y]$ and $\omega\in\Omega$, then  $\ttdeg_{\omega}(f):= \max\{\omega^T\alpha\;:\; \;f_{\alpha}\neq 0\}$
is  the \textit{degree} of $f$ \wrt\  $\omega$.
The \textit{initial form} or \textit{leading terms}  of $f$ w.r.t. $\omega\in\Omega$ is defined as
$$\ttlt_{\omega}(f):= \sum_{\alpha\in A}{f_{\alpha} y^{\alpha}} 
\quad \text{where}
\quad A:=\left\{\alpha\in\N^n \;:\; f_{\alpha}\neq 0,\, \omega^T\alpha=deg_{\omega}(f)\right\}.
$$
The \textit{initial ideal} of $\I$ w.r.t. $\omega$ is the set
$\left\langle \ttlt_{\omega}(\I)\right\rangle:=\left\langle\;\left\{ \ttlt_{\omega}(f)\;:\; f\in \I\right\}\;\right\rangle$.

\begin{defn} Let $\I\subseteq\K[Y]$ be some fixed ideal, let $\prec$ be some monomial order and $\omega\in\Omega$. 
We say that 
\begin{itemize}
\item $\omega$ \emph{represents} $\prec$ if 
$\left<\ttlt_{\omega}(\I)\right>=\left<\ttlt_\prec(\I)\right>$ holds.
\item $\prec$ \emph{refines} $\omega$, 
if for all pairs of monomials $m_1,m_2\in[Y]$ one has that
$\ttdeg_\omega(m_1)<\ttdeg_\omega(m_2)$ implies $m_1\prec m_2$.                                                                                                                               \end{itemize}
\end{defn}
\smallskip

Not all weight vectors $\omega$ induce a proper monomial order.
But using some monomial order as an additional  tie-breaker does yield an order:
\begin{defn}
Given an ideal $\I\subseteq\K[Y]$, a monomial order $\prec$, and some weight vector $\omega$ 
the monomial order $(\omega|\prec)$ is defined as follows: 

Let $\prec':=(\omega|\prec)$, then 
$$
m_1\prec' m_2 
\quad :\Leftrightarrow \quad 
\left\{\begin{array}{lll}
 &\deg_\omega(m_1)<\deg_\omega(m_2)&\\
\text{or}&
\deg_\omega(m_1)=\deg_\omega(m_2)&\text{and}\;\; m_1\prec m_2
\end{array}\right.
$$
\end{defn}
So $(\omega|\prec)$ corresponds to first (partially) ordering the monomials by $\ttdeg_\omega$ and using $\prec$ as a tie-breaker. Clearly, the order $(\omega|\prec)$ refines $\omega$.

\subsubsection{The Gr\"obner fan}

\begin{defn}\label{defn:groebner_cone}
Given an ideal $\I\subseteq \K[Y]$ and a monomial order $\prec$, we define the Gr\"obner cone of $\I$ \wrt\ $\prec$ by
$$ C_\prec(\I) := \mathrm{closure}\left(\;\{\omega\in\Omega\;:\; \left<\ttlt_\omega(\I)\right> = \left<\ttlt_\prec(\I)\right>\}\;\right)$$
where $\mathrm{closure}$ denotes the closure with respect to the standard topology in $\R^{n+m}$.  
\end{defn}
For complete information on Gr\"obner cones, we would like to refer to \cite{article::mora_robbiano}, 
here we repeat some facts of these cones, relevant to this paper:

Each Gr\"obner cone of $\I$ is a convex polyhedral cone with non-empty interior (see \cite{article::mora_robbiano}) and
the set of all Gr\"obner cones forms a polyhedral complex, namely the Gr\"obner fan $\mathcal{C}(\I):=\{C_\prec(\I)\;:\;\prec \text{is some monomial order}\}$.
\smallskip

Moreover, each Gr\"obner cone corresponds to some reduced Gr\"obner basis, i.e.,
 all monomial orders, which are represented by the weight vectors within the same Gr\"obner cone, will have the same reduced Gr\"obner basis.
This implies that $\I$ has only finitely many different Gr\"obner cones. 
Moreover, we obtain the following for a weight vector and a monomial order constructed from it:
\begin{lem}\label{lem:represent}
For a  weight vector $\omega\in\Omega$  and some order $\prec$ let $\prec_\omega:=(\omega|\prec)$.
With this one has $\omega\in C_{\prec_\omega}(\I)$.
\smallskip

Reversely, if $\omega  \in C_{\prec}(\I)$ holds, then 
$\ttlt_{\prec_\omega}(g)=\ttlt_{\prec}(g)$ holds for all $g\in\redGB(\I,\prec)$, which consequently implies
$\redGB(\I,\prec_\omega)=\redGB(\I,\prec)$.
\end{lem}

For a proof we refer to Lemma 2.15 and Corollary 2.11 in \cite{article::Fukuda_Jensen_Thomas_2}. 

\subsection{Geometry}
In the following we prove that some special set of weight vectors is star-shaped, to this end we recall the following:

\begin{defn}
 A set $S\subseteq\R^{n+m}$ is called \textit{star-shaped} with center $C\subseteq S$, 
if for any two points $s\in S$ and $c\in C$ the segment $\overline{ms}$ is contained in $S$.
\end{defn}

\subsection{Universal elimination orders}
In the following we assume $\I$ to be some ideal in $\K[X][U]$.
A class of monomial orders, 
which provides a reduced Gr\"obner basis for the elimination ideal $\I\cap\K[X]$ is the set of elimination orders; 
these orders are traditionally used to calculate the elimination ideal via Gr\"obner bases.

\begin{defn} \label{def::EO} 
A monomial order $\prec$ on $\K[X][U]$ is called \emph{\textbf{universal} elimination order} for $U$, if 
    $$\ttlt_{\prec}(f)\in \K[X] \quad\Rightarrow\quad f\in \K[X] \qquad\forall f\in \K[X][U].$$
\end{defn}

So a universal elimination order for $U$ will have to prefer \emph{any} $u$-variable over some $x$-variable. For example, an appropriate lexicographic order is a universal elimination order.
A universal elimination order can be used to calculate the GB of the elimination ideal for \emph{any} given ideal:
\begin{lem}
If $\prec$ is a universal elimination order, then for \emph{every} ideal $\I$, 
the set $\redGB(\I,\prec)\cap\K[X]$ is the reduced Gr\"obner basis of the elimination ideal $\I\cap \K[X]$ w.r.t. $\prec$.
\end{lem}
\noindent For a proof see \cite{article::Tran1}.

\medskip

\subsection{Ideal-specific elimination orders}
In contrast to universal elimination orders, in this paper we examine ideal-specific elimination orders,
which serve to eliminate variables only for the specifically given ideal:

\begin{defn} \label{def:I_EO} (Ideal-specific elimination orders and vectors)\\
Let $\I\subseteq \K[X][U]$ be an ideal and $\prec$ a monomial order with 
$$
\ttlt_{\prec}(g)\in \K[X] \quad\Rightarrow\quad g\in \K[X] \qquad\forall g\in \redGB(\I,\prec).
$$
\begin{enumerate}
\item Then $\prec$ is called \textit{$\I$-specific elimination order for the elimination of $U$}.
When clear which variables are to be eliminated we abbreviate this to \textit{$\I$-specific elimination order}, or just  $\I$-EO.
\item Any $\omega \in C_{\prec}(\I)$ is called \textit{$\I$-specific for the elimination of $U$} ($\I$-{EV}).
\end{enumerate}
\end{defn}
In the following we will always consider the elemination of the $u$-variables for ideals in $\K[X][U]$,
so all ideal-specific elimination orders and ideal-specific elimination vectors will be ideal-specific for the elimination of $U$. 
\medskip

The reduced Gr\"obner basis for an $\I$-EO yields a Gr\"obner basis for the elimination ideal:
\begin{lem}\label{lem:red_GB_of_I_EO}
 Let $\I\subset\K[X][U]$ be some fixed ideal. If $\prec$ is an $\I$-specific elimination order for the elimination of $U$, then the set $\redGB(\I,\prec)\cap\K[X]$ is the reduced Gr\"obner basis of the elimination ideal $\I\cap \K[X]$ \wrt $\prec$.
\end{lem}
\noindent For a proof see \cite{article::Tran}.
\smallskip

So any $\I$-EO yields a Gr\"obner basis suitable for the elemination of the variables $u_i$ from $\I$. 
But in contrast to \emph{universal} elimination orders, an $\I$-EO will in general not work for other polynomial ideals. However, any \emph{universal} elimination order is -by definition- also an $\I$-EO for any ideal $\I$.
\smallskip

For our proofs we use the following characterization for an $\I$-EO:
\begin{lem}
Let $\I\subset\K[X][U]$ be some fixed ideal. A monomial order is $\I$-EO for $U$ if and only if
\begin{equation}\label{eq:characterization}
 \ttlt_{\prec}\left(\I\cap \K[X]\right)=\ttlt_{\prec}(\I)\cap \K[X].
\end{equation}
\end{lem}
The implication ``$\subseteq$'' in  (\ref{eq:characterization}) can be directly seen,  for a complete proof of the converse we refer to \cite{article::Tran}.
\medskip

By Definition \ref{def:I_EO}, if $\prec$ is an $\I$-{EO}, then any weight vector in  the Gr\"obner cone $C_\prec(\I)$ is $\I$-EV.  
Now assume (conversely) that one finds some $\I$-EV $\omega$ in $C_\prec(\I)$ with $\omega\ne 0$.
In some of these cases, one can conclude that $\prec$ is an $\I$-EO, namely 
if $\omega$ is in the \emph{interior} of $C_\prec(\I)$  (see  \cite{article::Tran}) or if 
$\omega$ lies on a \emph{special} part of the boundary of $\Omega$:

%
%
\begin{lem}\label{lem:I_EO_on_boundary}
Let $\I\subseteq \K[X][U]$ be some ideal and $\prec$ some monomial order.
If $(\zero,\widetilde\omega_u)\in C_{\prec}(\I)$ holds for some $\widetilde\omega_u\in\R^m_{>0}$, then $\prec$ is an $\I$-EO.
\end{lem}
Lemma \ref{lem:I_EO_on_boundary} and its proof can be found in \cite{article::Tran1},
it is used to obtain the main result in \cite{article::Tran}.
Geometrically, this lemma proves that $\prec$ is $\I$-EO if its Gr\"obner cone 
intersects the \emph{relative interior} of a special face $\Omega_u$ of the polyhedron $\Omega$, where
$$\Omega_u:=\{\omega\in \R_{\ge 0}^{n+m}\;:\; \omega_1=0,\ldots,\omega_n=0 \}.$$

Since each $C_\prec(\I)$ containing some vector $(\zero,\widetilde\omega_u)$ with $\widetilde\omega_u >0$ comes from some $\I$-EO $\prec$, and since these Gr\"obner cones are closed, Lemma \ref{lem:I_EO_on_boundary} implies

\begin{cor}\label{cor:I_EO_on_boundary}
All vectors $\omega\in\F$ are $\I$-EVs.  
\end{cor}

\section{Main result}
Our main result is the following:
\begin{thm}\label{thm:main_thm}
Let $\I$ be a polynomial ideal in $\K[x_1,\ldots,x_n, u_1,\ldots, u_m]$.

The Gr\"obner cones of all $\I$-specific elimination orders for the elimination of $u_1,\ldots,u_m$ 
from $\I$ 
form a star-shaped region, whose center is the following face of $\R_{\ge 0}^{n+m}$:
$$\F:=\{\omega\in \R_{\ge 0}^{n+m}\;:\; \omega_1=0,\ldots,\omega_n=0 \}$$
\end{thm}

\begin{proof}
Let $\tau\in\F$, and let $\sigma$ be some $\I$-EV, i.e., one has $\sigma\in C_{\prec'}(\I)$ for some $\I$-EO $\prec'$. 
Here $\tau$ can be part of the relative boundary of $\F$, e.g. $\tau=\zero$ is possible. 
By Corollary \ref{cor:I_EO_on_boundary}, we know that $\tau$ is $\I$-EV and thus we have to show that all other points in the segment $[\sigma,\tau]$ are $\I$-EV, too.
So let $\omega:=\lambda\sigma+(1-\lambda)\tau$ with $\lambda\in(0,1)$.
\smallskip

Let $\prec:=(\sigma|\prec')$, then due to $\sigma\in C_{\prec'}(\I)$ one has $\redGB_\prec(\I)=\redGB_{\prec'}(\I)$ -- see Lemma \ref{lem:represent}.
Moreover the orders $\prec$ and $\prec'$ yield the same leading terms on all $g\in\redGB_\prec(\I)$ (Lemma \ref{lem:represent}) and so $\prec$ is $\I$-EO by Definition \ref{def:I_EO}.

\smallskip

We examine the monomial orders $\prec_\sigma:=(\sigma|\prec)$, $\prec_{\tau}:=(\tau|\prec)$, and $\prec_{\omega}:=(\omega|\prec)$ and show that $\prec_{\omega}$ is $\I$-EO, 
which together with $\omega\in C_{\prec_{\omega}}(\I)$ (see Lemma \ref{lem:represent}) shows that $\omega$ is $\I$-EV.
Note that one has $\prec_\sigma=\prec$ and thus $\prec_\sigma$ is $\I$-EO.
\medskip

Now we show that $\prec_{\omega}$ is $\I$-EO. 
Due to $\tau\in\F$ one has $\tau=(\tau_x,\tau_u)$ with $\tau_x=\zero$, implying that for $\omega=(\omega_x,\omega_u)$ one has 
$\omega_x=\lambda \sigma_x$ with $\lambda>0$. So $\prec_{\omega}=(\omega|\prec)$ and $\prec_\sigma=(\sigma|\prec)$ coincide on $\K[X]$.
This implies 
$$
\ttlt_\prec(\I\cap\K[X])=\ttlt_{\prec_\sigma}(\I\cap\K[X])=\ttlt_{\prec_{\omega}}(\I\cap\K[X]),
$$
we call this set $\LX$.
\smallskip

Assume now, that $\prec_\omega$ is \emph{not} $\I$-EO, i.e., 
$\ttlt_{\prec_\omega}(\I)\cap\K[X]\ne \LX$.
This implies $\ttlt_{\prec_\omega}(\I)\cap\K[X]\not\subseteq \LX$ since the reverse inclusion is always true.
So there must be some $g\in\I$ with $\ttlt_{\prec_\omega}(g)\in\K[X]\setminus \LX$.

Let $x^\alpha:=\ttlt_{\prec_\omega}(g)$ and $m:=\ttlt_{\prec_\sigma}(g)$, then $m\ne x^\alpha$. This holds, since $x^\alpha=m$ leads to the contradiction 
$x^\alpha=m\in \ttlt_{\prec_\sigma}(\I)\cap\K[X]=\LX$ (the latter holds since $\prec_\sigma$ is $\I$-EO).

We conclude $x^\alpha \prec_\tau m$, 
since $\tndeg_\tau(x^\alpha)=0\le\tndeg_\tau(m)$ holds by choice of $\tau$ where in case of ``$=$'' the tie-braker $\prec=\prec_\sigma$ yields $x^\alpha \prec_\sigma m$.
So in total we obtain
 \vspace{4pt}

\begin{tabular}{lll}
 1. &$x^\alpha{\prec_\sigma} m$ &(since $\ttlt_{\prec_\sigma}(g)=m$)\\
 2. &$m \prec_\omega x^\alpha$ &(since $\ttlt_{\prec_\omega}(g)=x^\alpha$)\\
 3. &$x^\alpha\prec_\tau m$ &(since $\tndeg_\tau(x^\alpha)=0$ and $x^\alpha \prec m$).
\end{tabular}
\vspace{4pt}

By the constructions of these monomial orders we conclude 
$$ \left.\begin{array}{lcl}
\tndeg_\sigma(x^\alpha)&\le& \tndeg_\sigma(m)\\
\tndeg_\tau(x^\alpha)&\le& \tndeg_\tau(m)\\
       \end{array}\right\} \quad\Rightarrow \quad
\tndeg_\omega(x^\alpha)\le \tndeg_\omega(m),
  $$
where ``$=$'' in the last inequality implies ``='' in all three inequalties, leading to $x^\alpha\prec m$ (due to $x^\alpha\prec_\sigma m$).
This yields the contradiction $x^\alpha\prec_\omega m$, showing that $g$ cannot exist and thus $\prec_\omega$ is $\I$-EO.
\end{proof}


\section{The geometry of ideal-specific elimination vectors}

In this section we prove two further geometrical properties of the set of all ideal-specific elimination vectors for a given ideal.

Theorem \ref{thm:main_thm} shows that for a given ideal $\I$, the $\I$-specific elimination vectors form a set that is star-shaped. Here we prove that this set in general is non-convex. 
Finally, we prove by example that an ideal $\I$ can have an $\I$-specific elimination order $\prec$, 
whose Gr\"obner cone $C_\prec(\I)$ intersects the exterior of the Gr\"obner fan of $\I$ in the origin only.
\smallskip

\subsection{Cones in the interior}

\begin{lem}\label{lem:cones_in_interior}
There are ideals $\I\subset\K[X][U]$ 
which have an $\I$-EO $\prec$ whose Gr\"obner cone $C_\prec(\I)$ intersects the boundary of the Gr\"obner fan in the origin $0$ only.
 \end{lem}

\begin{proof}\label{eg:interior}\normalfont
Consider the following ideal $\I=\left\langle x^2-1,xu^2-x-u\right\rangle\subseteq \K[x][u]$. 
There are exactly three different reduced Gr\"obner bases of $\I$, which correspond to the three Gr\"obner cones of the Gr\"obner fan:
$$\begin{array}{rcl}
G_1&=& \left\{\textbf{x}^2 - 1, \textbf{u}^2 - xu - 1\right\}\\
G_2&=& \left\{\textbf{x}^2 - 1, \textbf{xu} - u^2 + 1, \textbf{u}^3 -2u-x\right\}\\
G_3&=& \left\{\textbf{x} + 2u - u^3, \textbf{u}^4 - 3u^2 + 1\right\}
\end{array}$$
Here the leading terms are given in bold letters.

\begin{figure}[!h]
\begin{center}
\begin{picture}(200,130)

 \put(0,0){\includegraphics[scale=0.27]{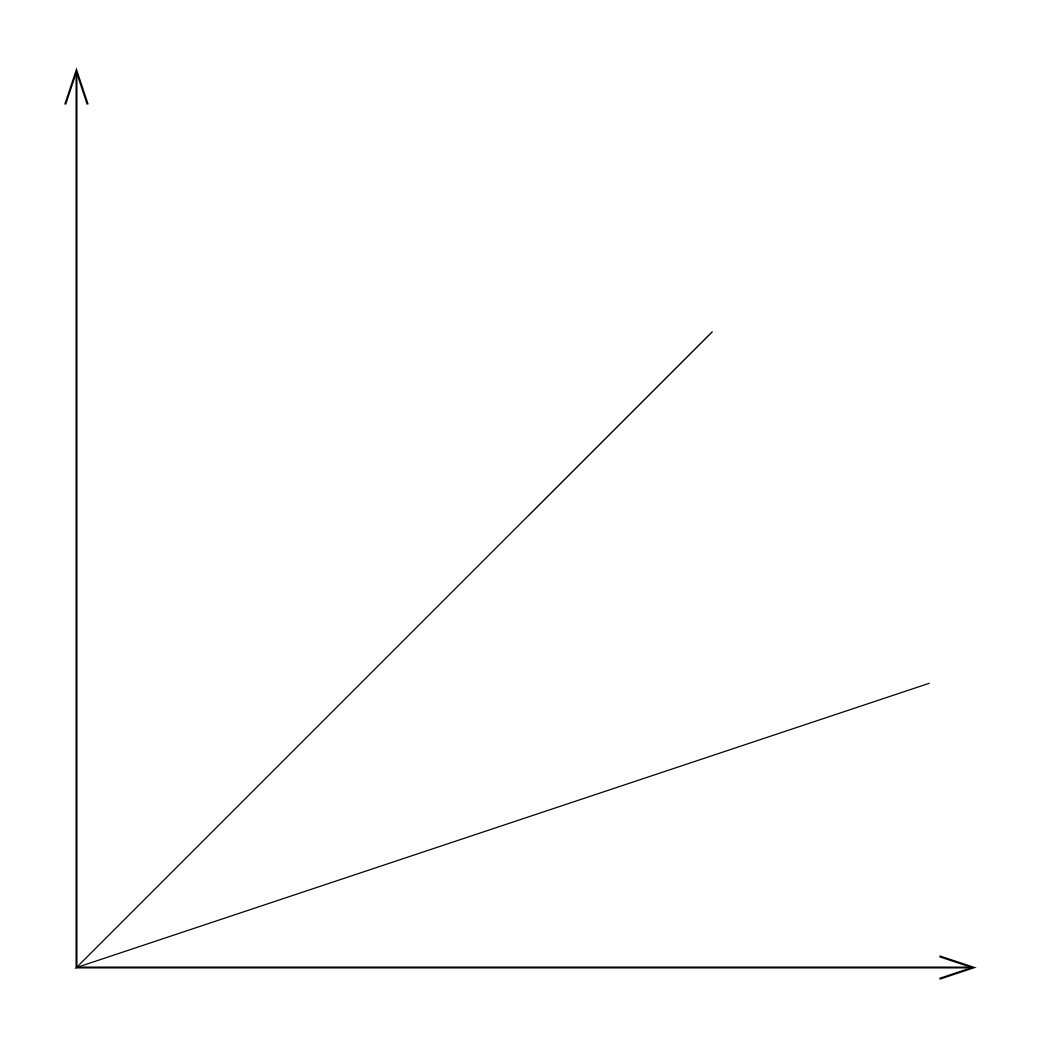}}

\put(118,0){{$\omega_x$}}
\put(-8,123){{$\omega_u$}}
\put(-6,60){{$\F$}}
\put(100,20){{$C_3$}}
\put(80,50){{$C_2$}}
\put(30,90){{$C_1$}}
\end{picture}
 \end{center}
\caption{Gr\"obner fan of $\I=\left\langle x^2-1,xu^2-x-u\right\rangle$}\label{fig:interior}
\end{figure}
For $i=1,2,3$ let $C_i$ be the Gr\"obner cones corresponding to the Gr\"obner basis $G_i$ and let $\prec_i$ be some corresponding monomial order.

Examining the polynomials in $G_1$ and $G_2$ in respect to Definition \ref{def:I_EO}, one observes that $\prec_1$ and $\prec_2$ are $\I$-specific elimination orders for the elimination of $u$.
We now check that the cone $C_2$ must be in between the cones $C_1$ and $C_3$ (see Figure \ref{fig:interior}).

It is easy to check that for $\overline{\omega}:=(1,0)^T$ one has  
$\ttlt_{\overline{\omega}}(G_3)=\ttlt_{\prec_3}(G_3)$ 
and $\ttlt_{\overline{\omega}}(G_i)\ne\ttlt_{\prec_i}(G_i)$ for $i=1,2$.
This implies that $\overline\omega\in C_3$ holds. In the same way one proves $(0,1)^T\in C_1$.

Since the Gr\"obner fan considered here is two-dimensional, $C_2$ must thus be in between $C_1$ and $C_3$.
This proves that for $\I$, there is indeed an $\I$-EO ($\prec_2$) whose Gr\"obner cone $C_2$ intersects the boundary of the Gr\"obner fan in $(0,0)^T$ only.
\end{proof}

\subsection{Non-convexity}

It seems intuitive at first sight that the set of all $\I$-EVs should be convex, but this is in general not true.

\begin{exe}\label{eg:nonconvex}\normalfont
Let $\I:=\left\langle x + u + v, x^2 -1\right\rangle\subseteq \K[x][u, v]$ 
and set $\sigma:=(9, 12, 0)^T$, $\tau:=(9, 0, 10)^T\in\Omega$, and $\omega:=\frac{1}{2}\sigma + \frac{1}{2}\tau=(9,6,5)^T\in\overline{\sigma \tau}$. Let $\prec_{\sigma}, \prec_{\tau}$ and $\prec_{\omega}$ be monomial orders refining $\sigma, \tau$, and  $\omega$ respectively.\\ 

\begin{figure}[!htbp]
\setlength{\unitlength}{0.5pt}
\begin{center}

\begin{picture}(200,245)\hspace*{-48pt}
\put(-30,103){\includegraphics[scale=0.2]{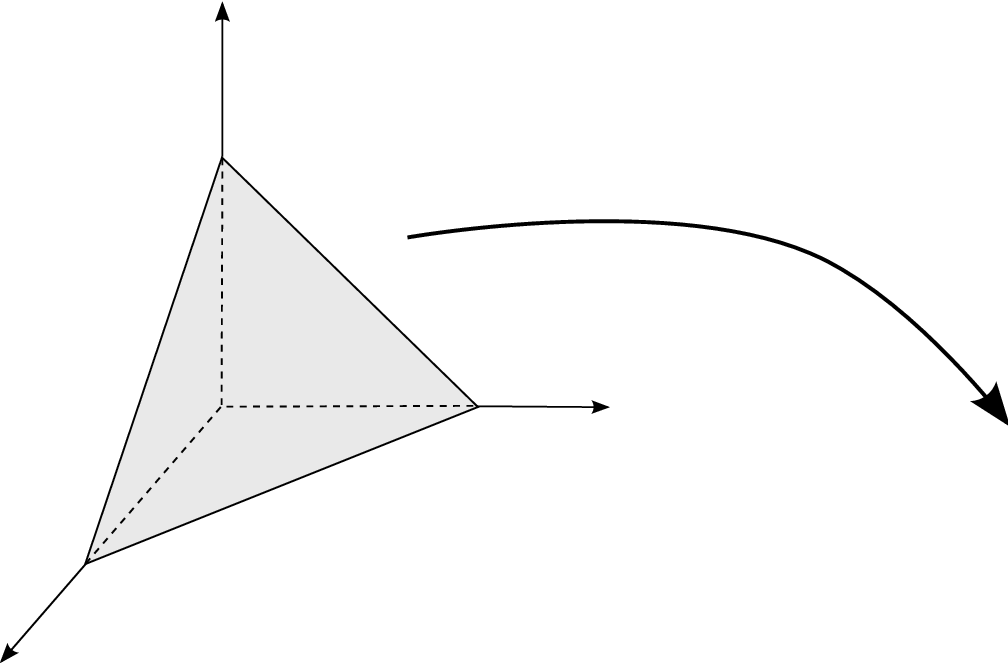}}
\put(-50,113){{$\omega_u$}}
\put(78,158){{$\omega_v$}}
\put(20,218){{$\omega_x$}}

\put(95,0){\includegraphics[scale=0.15]{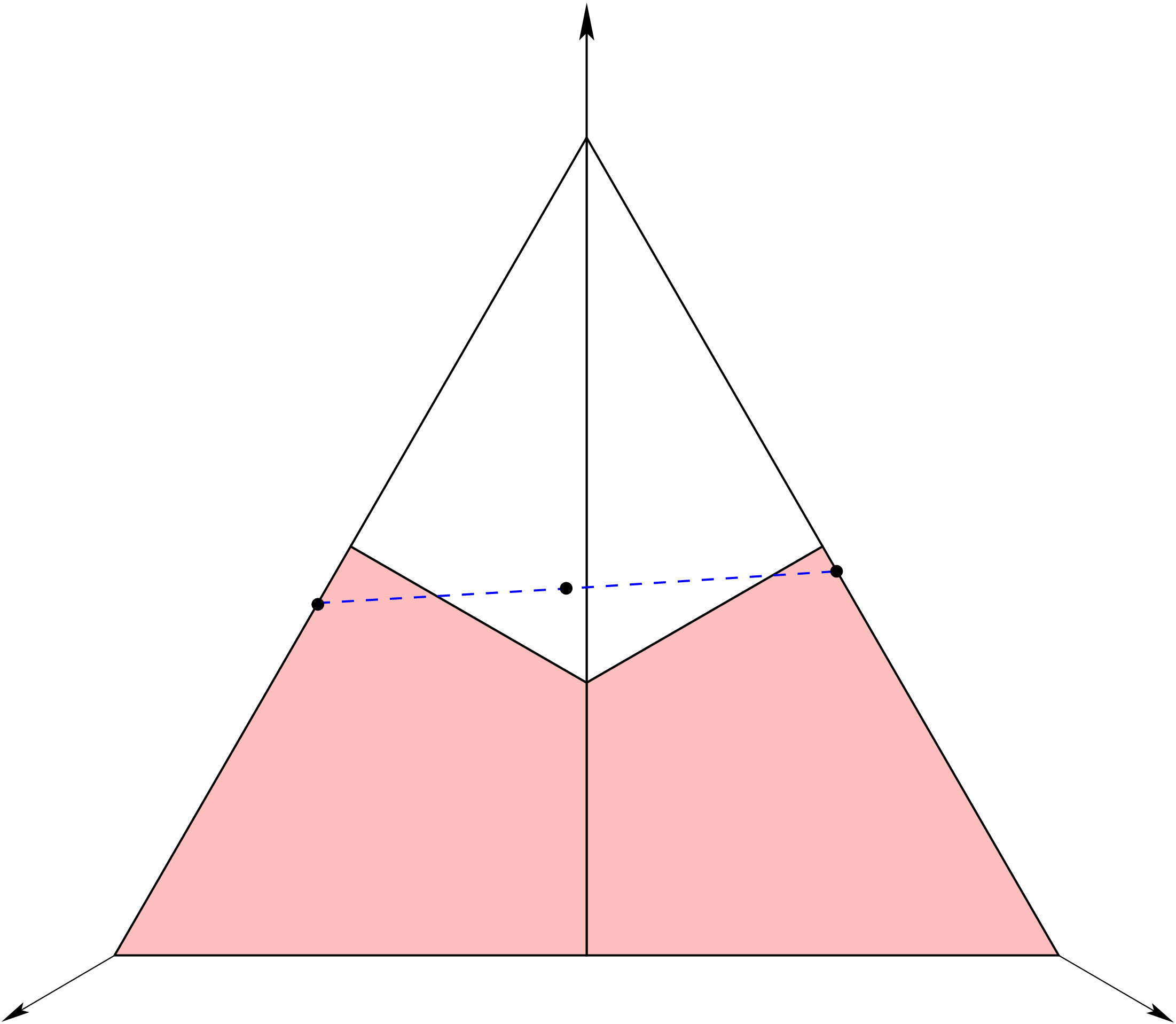}}
\put(260,250){{$\omega_x$}}
\put(75,10){{$\omega_u$}}
\put(405,10){{$\omega_v$}}

\put(168,114){{$\sigma$}}
\put(240,123){{$\omega$}}
\put(325,124){{$\tau$}}
\put(170,50){{$C_{\prec_{\sigma}}$}}
\put(215,153){{$C_{\!\prec_{\omega}}$}}
\put(275,50){{$C_{\prec_{\tau}}$}}
\end{picture}\setlength{\unitlength}{1pt}
\caption{Gr\"obner fan for $\I=\left\langle x+u+v,x^2-1\right\rangle$ }\label{fig:nonconvex}
\end{center}
\end{figure}

Quick calculation shows that the reduced Gr\"obner bases w.r.t. $\prec_{\sigma}$ and $\prec_{\tau}$ are the following
$$\begin{array}{ll}
\redGB(\I,\prec_{\sigma}) & =\left\{\textbf{u} + x+ v, \textbf{x}^2 - 1\right\},\\
\redGB(\I,\prec_{\tau}) &= \left\{\textbf{v} + x + u, \textbf{x}^2 - 1\right\}.
\end{array}$$
So by Definition \ref{def:I_EO}, both $\sigma$ and $\tau$ are $\I$-specific elimination vectors for elimination of the variables $u$ and $v$.
The reduced Gr\"obner basis w.r.t  $\prec_{\omega}$ is
$$\redGB(\I,\prec_{\omega})=\left\{\textbf{x} + u + v, \textbf{u}^2 + 2uv + v^2 - 1\right\}.$$
Since one has $\ttlt_{\prec_{\omega}}(x+u+v)=x\in\K[x]$ but $x+u+v\notin \K[x]$, by Definition \ref{def:I_EO},  $\prec_{\omega}$ can not be $\I$-specific for the elimination of $u$ and $v$.

Figure \ref{fig:nonconvex} depicts the Gr\"obner fan of $\I$ (intersected with some appropriate hyperplane) together with $\sigma, \omega$ and $\tau$. 
The highlighted Gr\"obner cones $C_{\prec_\sigma},C_{\prec_\tau}$ correspond to the $\I$-EOs $\prec_{\sigma}$ and $\prec_{\tau}$.
\end{exe}

\newpage

\section{Improving the elimination algorithm by Tran}\label{algo:elim_algo_tran}

The algorithm of Tran (Algorithm 1 in \cite{article::Tran}) 
calculates a Gr\"obner basis for the elimination ideal of by means of a generic Gr\"obner walk.

\subsection{Generic Gr\"obner walks}

A \emph{generic} Gr\"obner walk ``walks'' along some generic segment $\overline{\sigma\tau}\subset\Omega$. Such a segment is called generic if 
\begin{itemize}
 \item $\overline{\sigma\tau}$ only passes through the \emph{interior} of intermediate Gr\"obner cones or through 
\emph{interior points of their facets} and 
\item $\sigma$ is part of the \emph{interior} of some 
$C_{\prec_{\mathrm{start}}}(\I)$.
\end{itemize}

The walk starts with the (hopefully easy to compute) Gr\"obner Basis $G_0$ of $\mathcal I$ \wrt\ $\prec_0:=\prec_{\mathrm{start}}$.
Then sequentially, starting from $C_{\prec_0}(\I)$ for every cone $C_{\prec_k}(\I)$ through which $\overline{\sigma\tau}$ passes,
the intermediate GB $G_k$ \wrt\ $\prec_{k}$ is calculated. 
This can be done effectively by converting the previously calculated $G_{k-1}$ into $G_k$.
Each such basis-conversion from one intermediate GB into the next is (in general) relatively cheap computationwise, 
keeping the overall amount of necessary calculations relatively low (see \cite{article::AGK3}).

The walk terminates returning $G_\ell$ when reaching a cone $C_{\prec_{\ell}}(\I)$ containing $\tau$.

\subsection{Improvement to Tran's Stopping criterion}

The algorithm of Tran (Algorithm 1 in \cite{article::Tran}) 
which calculates a Gr\"obner basis for the elimination ideal by means of a Gr\"obner walk can be \emph{slightly} improved, by changing the termination criterion:
\smallskip

Tran's algorithm performs a Gr\"obner walk towards some $\tau$ in the relative interior of $\Omega_u$, i.e., a point  $\tau=(\tau_x,\tau_u)$ with $\tau_x=\zero$ and $\tau_u\in\R^m_{>0}$.
As a stopping criterion Tran uses Lemma \ref{lem:I_EO_on_boundary}, which states:
If some intermediate Gr\"obner cone $C_{\prec_k}(\I)$ contains $\tau$,
then the corresponding $\prec_k$ is an ideal-specific elimination order for the elemination of $U$. 
Tran then sets his Gr\"obner walk to terminate when reaching such a cell. 

Note the following: Since $\tau\in\Omega_u$ is part of the boundary of $\Omega$, in such a particular case $C_{\prec_k}(\I)$ intersects the boundary of $\Omega$ in more than just the origin.
In contrast, in Lemma \ref{lem:cones_in_interior} we prove that there are ideals $\I$, for which there are $\I$-specific elimination orders, whose Gr\"obner cones intersect the boundary of the Gr\"obner fan in \emph{just the origin}.
In this regard, our Algorithm \ref{algo:improvedTran} is an improvement of Tran's version.
 \medskip

\fbox{\begin{minipage}{0.98\textwidth}
\begin{algo}\label{algo:improvedTran}(Improved elimination algorithm)\\[2pt]
\sffamily
\begin{tabular}{lllp{11cm}}
\textbf{Input}  & \multicolumn{3}{p{14cm}}{$F=\{f_1,\ldots,f_\ell\}\subseteq \K[x_1,\ldots,x_n][u_1,\ldots,u_m]=\K[X][U]$}\\[4pt]
& \multicolumn{3}{p{14cm}}{$\tau\in\Omega$, where	%
	$\tau^T=(0^T,\tau_u^T)$ with $\tau_u\in\R^m_{>0}$,}\\	%
& \multicolumn{3}{p{14cm}}{$\sigma\in\Omega$, such that $\overline{\sigma\,\tau}$ is generic,}\\	%

& \multicolumn{3}{p{14cm}}{$\prec_{\sigma}, \prec_{\tau}$  refining $\sigma$ resp. $\tau$.}\\[8pt]
\textbf{Output} & \multicolumn{3}{p{14cm}}{$G\subseteq \K[X]$,  reduced Gr\"obner basis %
	of $\left\langle F\right\rangle\cap \K[X]$ }\\
& \multicolumn{3}{p{14cm}}{\wrt\ some $\left\langle F\right\rangle$-specific EO for $U$.}\\[8pt]
\textbf{Init} & \multicolumn{3}{p{14cm}} {%
	Calculate reduced start GB\ \  $G_0:=\redGB(\left\langle F\right\rangle,\prec_{\sigma})$ }\\
& \multicolumn{3}{p{14cm}} { 
	$k:=0$,\hspace{20pt}
	$\I:=\left\langle F\right\rangle$,\hspace{20pt}
	$\omega_0:=\sigma$,\hspace{20pt}
	$\prec_0:=\prec_{\sigma}$}\\[4pt]
\textbf{Step 1} & \multicolumn{3}{p{14cm}} {IF $\prec_k$ is $\I$-EO RETURN $G:=G_k\cap \K[X]$}\\[4pt]
\textbf{Step 2} &\multicolumn{3}{l}{\rule[-4pt]{0.5pt}{8pt}\rule[4pt]{2pt}{0.5pt}\hspace{2pt}GB-walk: change cell}        \\
                &\rule[-7pt]{0.5pt}{18pt}\rule[4pt]{4pt}{0.5pt}\hspace{2pt}(2.1) & \multicolumn{2}{l}{$k:=k+1$}\\
                &\rule[-7pt]{0.5pt}{18pt}\rule[4pt]{4pt}{0.5pt}\hspace{2pt}(2.2) & \multicolumn{2}{p{14,5cm}}{Find next weight vector %
	$\omega_k\in\overline{\sigma\tau}$, }\\
                &\rule[-7pt]{0.5pt}{18pt}       & \multicolumn{2}{p{14,5cm}}{\hspace{10pt}s.t. $\overline{\omega_{k-1}\,\omega_k}=\overline{\omega_{k-1}\,\tau}\cap C_{\prec_{k-1}}(\I)$.}\\ 
                &\rule[-7pt]{0.5pt}{18pt}\rule[4pt]{4pt}{0.5pt}\hspace{2pt}(2.3)& \multicolumn{2}{l}{Set $\prec_{k}:=(\omega_k | \prec_\tau)$.}\\
                &\rule[-7pt]{0.5pt}{18pt}\rule[4pt]{4pt}{0.5pt}\hspace{2pt}(2.3)& \multicolumn{2}{l}{Convert $G_{k-1}$ into Gr\"obner basis $G_k$ \wrt\ $\prec_k$.}\\
                &\rule[4pt]{0.5pt}{11pt}\rule[4pt]{4pt}{0.5pt}\hspace{2pt}(2.4)& \multicolumn{2}{p{14cm}}{Interreduce $G_k$.}\\[4pt]
\textbf{Step 3} & \multicolumn{3}{p{14cm}} {GOTO Step 1} 
\end{tabular}
\end{algo}
\end{minipage}}
\medskip

\begin{rem} Algorithm \ref{algo:improvedTran} as stated above, is in fact Tran's Algorithm in \cite{article::Tran} -- the only difference being the (refined) stopping criterion in Step 1. In Tran's original version of Algorithm 1, his stopping criterion in Step 1 reads: 
\begin{center}
``\ {\sffamily IF $\tau\in C_{\prec_k}(\I)$ RETURN  $G:=G_k\cap \K[X]$.}\ '' 
\end{center}
\end{rem}

\begin{thm}
Algorithm \ref{algo:improvedTran} terminates and is correct.
\end{thm}
\begin{proof}
In each $\omega_k$ the section $\overline{\sigma\, \tau}$ crosses from some Gr\"obner cone into another. 
Since there are only finitely many such transition-points $\omega_k$, the algorithm 
can only perform a finite number of steps.
\smallskip

The algorithm terminates as soon as it passes some $\omega_\ell$ where  $\prec_\ell$ is $\I$-EO. 
Such a point $\omega_\ell$ must exist, due to the following: 

One has $\tau \in C_{\prec_\tau}(\I)$ and thus $\overline{\sigma\tau}\cap C_{\prec_\tau}(\I)\ne\emptyset$.
Let $\omega$ be the first point on $\overline{\sigma\tau}$ that is in $C_{\prec_\tau}$, and assume the algorithm does not terminate on any point in $\overline{\sigma\omega}\setminus\{\omega\}$.
Then $\omega=\omega_k$ holds for some $k$, since
either $\omega=\sigma=\omega_0$ holds or $\omega$ is on the boundary of some $C_{\prec_{k-1}}(\I)$ of the examined $\omega_{k-1}$. 
In either case, the algorithm will calculate $\redGB_{\prec_\omega}(\I)$ for $\prec_\tau=(\omega|\prec_\tau)$ and then terminate, since 
$\omega\in C_{\prec_\tau}$ and thus
$\prec_\tau$ is $\I$-EO by Lemma \ref{lem:I_EO_on_boundary}.
\smallskip

With $\prec_\ell$ being an $\I$-EO, due to Lemma \ref{lem:red_GB_of_I_EO} $G_\ell\cap\K[X]$ is a Gr\"obner basis for $\left<F\right>\cap\K[X]$ (see Lemma \ref{lem:red_GB_of_I_EO}).

\end{proof}
%
%

\begin{exe}
In Lemma \ref{lem:cones_in_interior} we present the ideal $\I=\left\langle x^2-1,xu^2-x-u\right\rangle$
where the difference in stopping criteria actually matters -- see Figure \ref{fig:walk}.
For $\I$ there are three reduced Gr\"obner bases, namely (leading terms in bold letters)
$$\begin{array}{rcl}
\redGB_{\prec_0}(\I)  &=& \left\{\textbf{x}  - u^3 + 2u,\, \textbf{u}^4 - 3u^2 + 1\right\}\\
\redGB_{\prec_1}(\I)     &=& \left\{\textbf{x}^2 - 1,\, \textbf{xu} - u^2 + 1,\, \textbf{u}^3 -2u-x\right\}\\
\redGB_{\prec_2}(\I)&=& \left\{\textbf{x}^2 - 1,\, \textbf{u}^2 - xu - 1\right\}
\end{array}$$
We now start our walk with the Gr\"obner basis $\redGB_{\prec_0}(\I)$
at $\sigma =(7,1)$, in the interior of $C_{\prec_0}$. We then walk towards $\tau=(0,8)\in C_{\prec_2}$. 
With this setting, our algorithm stops after reaching $\omega_1=(6,2)$ with the Gr\"obner basis $\redGB_{\prec_1}(\I)$
while Tran's algorithm stops after reaching $\omega_2=(4,4)$ with the Gr\"obner basis $\redGB_{\prec_2}(\I)$.
So Tran's algorithm calculates one more basis conversion than our Algorithm 1.

\begin{figure}[!h]
\begin{center}
\begin{picture}(200,140)

\put(-65,0){
\put(10,10){\includegraphics[scale=0.35]{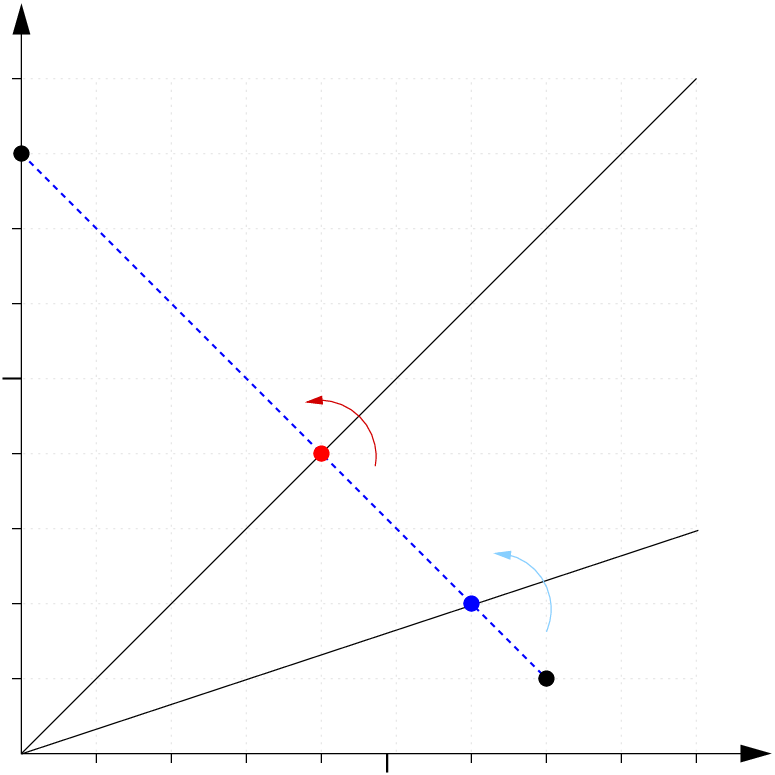}}

\put(118,0){{$\omega_x$}}
\put(-3,128){{$\omega_u$}}
\put(-6,60){{$\F$}}

\put(65,60){\color{red} $\scriptstyle\omega_2$ }
\put(92,34){\color{blue} $\scriptstyle\omega_1$ }

\put(105,25){\color{black} $\scriptstyle\sigma$ }
\put(15,115){\color{black} $\scriptstyle\tau$ }

\put(130,20){{$C_{\prec_0}(\I)$}}
\put(100,70){{$C_{\prec_1}(\I)$}}
\put(50,110){{$C_{\prec_2}(\I)$}}
}

\put(80,90){
\fbox{\begin{minipage}{190pt}
\begin{tabular}{lll}
$\scriptstyle\sigma=$&$\scriptstyle(7,1)$& \small Start of the walk. \\ 
$\scriptstyle\tau=$&$\scriptstyle(0,8)$&  \small  Target of the walk, $\prec_\tau :=\prec_2$.\\ 

$\scriptstyle{\color{blue}\omega_1}=$&$\scriptstyle(6,2)$&  \small Endpoint of our algorithm, \\
                                                        &&\small order is  $\prec_1=(\omega_1|\prec_\tau)$\\ 

$\scriptstyle{\color{red}\omega_2}=$&$\scriptstyle(4,4)$&  \small Endpoint of Tran's algorithm, \\
                                                        &&\small order is  $\prec_2=\prec_\tau=(\omega_2|\prec_\tau)$\\

\end{tabular}
      \end{minipage}}}
\end{picture}
 \end{center}
\caption{Gr\"obner walk for  $\I=\left\langle x^2-1,xu^2-x-u\right\rangle$}\label{fig:walk}
\end{figure}
\end{exe}

\section{Conclusion and Outlook}

The work of Tran in both \cite{article::Tran} and \cite{article::Tran1} provides proper algorithms to make use of Gr\"obner walks in the elemination of variables from polynomial ideals. Tran's approach even simplifies perturbing the corresponding walk in order to obtain a generic walk.
\medskip

Our results refine this work. We provide a geometric interpretation of the set of ideal-specific elimination vectors. More precisely we prove that these weight vectors form a star-shaped region. 
More surprisingly, we show that the corresponding region in general is not convex.
\medskip

Finally, we redefine Tran's stopping criterion and show that this yields some improvement over Tran's original stopping criterion. Tran's criterion stops the walk when reaching a Gr\"obner cone containing the target weight vector $\tau$, which in turn is part of the boundary of the Gr\"obner fan.
In contrast to this, we show that for some polynomial ideals one can terminate the Gr\"obner walk in some ``interior'' Gr\"obner cone, namely a cone whose intersection with the boundary of the Gr\"obner fan is just the origin.
Whether this improvement yields a significant improvement for the average running time of Tran's algorithm 
is not clear, and should be subject to further research.
\bigskip

A possible improvement to our work would be to check wether the star-shapedness of the region of interest gives rise to cleverly changing the direction of the walk, leading to a more efficient zig-zag-walk. 
More precisely one would like to answer the following:

If, in some step of the Gr\"obner walk, the current Gr\"obner cone borders (via a facet) to some cone of an ideal-specific elemination order, one could of course terminate the walk with a single step.
Is it possible to cheaply determine such situations from the current Gr\"obner basis? 
\bigskip

\section{Thanks}
We are deeply thankful for our very supportive referees!

\label{the_bibl}
  \bibliographystyle{plain}

\bibliography{bibl}

\begin{thebibliography}{10}

\bibitem{article::AGK3}
B.~Amrhein, O.~Gloor, and W.~K\"uchlin.
\newblock On the walk.
\newblock {\em Theoretical Computer Science}, 187:179--202, 1997.

\bibitem{phd::buchberger}
B.~Buchberger.
\newblock {\em Ein {A}lgorithmus zum {A}uffinden der {B}asiselemente des
  {R}estklassenringes nach einem nulldimensionalen {P}olynomideal}.
\newblock PhD thesis, Leopold-Franzens-Universit\"at Innsbruck, 1965.

\bibitem{proceedings::Buchberger}
B.~Buchberger.
\newblock Applications of {G}r{\"o}bner bases in non-linear computational
  geometry.
\newblock In {\em Proceedings of the International Symposium on Trends in
  Computer Algebra}, pages 52--80, London, UK, 1988. Springer-Verlag.

\bibitem{article::Collart97convertingbases}
S.~Collart, M.~Kalkbrener, and D.~Mall.
\newblock Converting bases with the {G}r\"obner walk.
\newblock {\em Journal of Symbolic Computation}, 24:465--469, 1997.

\bibitem{book::cox_little_oshea}
D.~Cox, J.~Little, and D.~O'Shea.
\newblock {\em Ideals, Varieties, and Algorithms}.
\newblock Undergraduate Texts in Mathematics. Springer-Verlag, New York, 1992.

\bibitem{article::Fukuda_Jensen_Thomas}
K.~Fukuda, A.~N. Jensen, N.~Lauritzen, and R.~Thomas.
\newblock The generic {G}r{\"o}bner walk.
\newblock {\em Journal of Symbolic Computation}, 42:298--312, March 2007.

\bibitem{article::Fukuda_Jensen_Thomas_2}
K.~Fukuda, A.~N. Jensen, and R.~Thomas.
\newblock Computing {G}r{\"o}bner fans.
\newblock {\em Mathematics of Computation}, 76:2189--2212, 2007.

\bibitem{proceedings::Kalkbrenner}
M.~Kalkbrenner.
\newblock Implicitization of rational parametric curves and surfaces.
\newblock In {\em AAECC-8, {P}roceedings of the 8th {I}nternational {S}ymposium
  on {A}pplied {A}lgebra, {A}lgebraic {A}lgorithms and {E}rror-{C}orrecting
  {C}odes}, AAECC-8, pages 249--259, London, UK, 1991. Springer-Verlag.

\bibitem{article::mora_robbiano}
T.~Mora and L.~Robbiano.
\newblock The {G}r{\"o}bner fan of an ideal.
\newblock {\em Journal of Symbolic Computation}, 6:183--208, December 1988.

\bibitem{book::sal}
G.~Salmon.
\newblock {\em Lessons {I}ntroductory to the {M}odern {H}igher {A}lgebra, fifth
  ed.}
\newblock Chelsea Publishing Company, Bronx, New York, 1964.

\bibitem{article::sed}
T.W. Sederberg, D.C. Anderson, and R.~N. Goldman.
\newblock Implicit representation of parametric curves and surfaces.
\newblock {\em Computer Vison, Graphics and Image Processing}, 28:72--84, 1984.

\bibitem{book::syl}
J.~Sylvester.
\newblock {\em The collected mathematical papers of {J}ames {J}oseph
  {S}ylvester}.
\newblock Cambridge University Press, Cambridge, England, 1904.

\bibitem{article::Tran1}
Q.-N. Tran.
\newblock Efficient {G}r{\"o}bner walk conversion for implicitization of
  geometric objects.
\newblock {\em Journal of Computer Aided Geometric Design}, 21:837--857,
  November 2004.

\bibitem{article::Tran}
Q.-N. Tran.
\newblock A new class of term orders for elimination.
\newblock {\em Journal of Symbolic Computation}, 42:533--548, 2007.

\end{thebibliography}

\end{document}